\documentclass[a4paper,11pt]{amsart}

\usepackage{amssymb,amsthm,amsmath,amscd,amsfonts,bbm,mathrsfs}

\usepackage[cmtip,all]{xy}

\theoremstyle{plain}
\newtheorem{Lemma}{Lemma}
\newtheorem{Thm}[Lemma]{Theorem}
\newtheorem*{Thm*}{Theorem}
\newtheorem{Prop}[Lemma]{Proposition}
\theoremstyle{definition}
\newtheorem{Def}[Lemma]{Definition}
\theoremstyle{remark}
\newtheorem{Remark}[Lemma]{Remark}

\numberwithin{Lemma}{section}
\numberwithin{equation}{section}

\newcommand{\BBB}{\mathcal{B}}
\newcommand{\CCC}{\mathcal{C}}
\newcommand{\FFF}{\mathcal{F}}
\newcommand{\PPP}{\mathscr{P}}

\newcommand{\Fa}{\mathfrak{a}}
\newcommand{\Fb}{\mathfrak{b}}
\newcommand{\Fm}{\mathfrak{m}}
\newcommand{\FS}{\mathfrak{S}}

\newcommand{\II}{{\mathbb{I}}}
\newcommand{\NN}{{\mathbb{N}}}
\newcommand{\WW}{{\mathbb{W}}}

\DeclareMathOperator{\id}{id}
\DeclareMathOperator{\Hom}{Hom}
\DeclareMathOperator{\Ker}{Ker}

\begin{document}

\title{A note on Vasiu-Zink windows}
\author{Eike Lau}
\date{\today}
\address{Fakult\"{a}t f\"{u}r Mathematik,
Universit\"{a}t Bielefeld, D-33501 Bielefeld, Germany}
\email{lau@math.uni-bielefeld.de}

\begin{abstract}
We propose a notion of frames and windows that allows
an alternative proof of the Vasiu-Zink classification 
of $p$-divisible groups over ramified complete regular
local rings by their Breuil windows.
\end{abstract}

\maketitle

\section*{Introduction}

Let $k$ be a perfect field of characteristic $p\geq 3$.
In \cite{Vasiu-Zink}, Vasiu and Zink 
consider the regular local rings
$\FS=W(k)[[t_1,\ldots,t_r,u]]$ and $R=\FS/E\FS$, where
$$
E=u^e+a_{e-1}u^{e-1}+\cdots+a_0
$$
such that all $a_i\in W(k)[[t_1,\ldots,t_r]]$ are
divisible by $p$ and $a_0/p$ is a unit.
For $a\in\NN$ they consider also 
$\FS_a=\FS/u^{ae}\FS$ and $R_a=R/p^aR$
and write $\FS_\infty=\FS$ and $R_\infty=R$.
They define a category of Breuil windows
relative to $\FS_a\to R_a$ and a compatible
system of functors
$$
\kappa_a:(\text{Breuil windows rel.\ $\FS_a\to R_a$})
\to(\text{Dieudonn\'e displays over $R_a$}).
$$ 
Here Dieudonn\'e displays are equivalent to 
$p$-divisible groups by \cite{Zink-DDisp}.

\begin{Thm*}
The functor $\kappa_a$ is an equivalence
for all $a\in\NN\cup\{\infty\}$.
\end{Thm*}

For $a=1$ this is proved in \cite{Zink-Windows},
while $a=\infty$ is the main result of \cite{Vasiu-Zink}.
In this note we show that deformations from $a$ to $a+1$
of Breuil windows and of Dieudonn\'e displays are equivalent
because both are classified by lifts of the Hodge filtration.
With appropriate definitions, the known proof 
for Dieudonn\'e displays (recalled below) covers both cases. 
Technically, the main point is to separate the formalism
of $f_1$ from divided power constructions.

By induction it follows that $\kappa_a$ is an equivalence for
finite $a$; the case $a=\infty$ follows by passing to
the projective limit. Since the initial
case $a=1$ can be shown similarly,
this proof is essentially self-contained.

The author thanks Th.~Zink for pointing out 
an error in an earlier version.


\section{Frames and windows}

In this section $p=2$ is allowed.
The following notions of frames and windows do 
not coincide with the definitions in \cite{Zink-Windows}.

\begin{Def}
\label{DeFrame}
A frame is a quintuple
$\FFF=(S,I,R,f,f_1)$ where $S$ is a ring,
$R=S/I$ a quotient ring, $f$ an endomorphism of $S$,
and $f_1:I\to S$ an $f$-linear homomorphism
of $S$-modules.
We require that $S$ is local and that $f_1(I)$
generates $S$ as an $S$-module.
(In all examples actually $1\in f_1(I)$.)
\end{Def}

If $\FFF$ is a frame,
for an $S$-module $M$ we write $M^{(1)}=S\otimes_{f,S}M$,
and for an $f$-linear homomorphism of $S$-modules $g:M\to N$ 
we denote $g^\sharp:M^{(1)}\to N$ its linearisation,
$g^\sharp(s\otimes m)=sg(m)$.
There is a unique element
$\pi\in S$ such that $f(a)=\pi f_1(a)$
for $a\in I$.
Namely, if $f_1^\sharp(t)=1$ then $\pi=f^\sharp(t)$.

\begin{Def}
\label{DefWin}
A window over a frame $\FFF$ is a quadruple 
$\PPP=(P,Q,F,F_1)$ where $P$ is a finite free
$S$-module, $Q\subseteq P$ is a submodule, 
$F:P\to P$ and $F_1:Q\to P$ are $f$-linear
homomorphisms of $S$-modules, such that
\begin{enumerate}
\item $IP\subseteq Q$ and $P/Q$ is free over $R$, 
\item $F_1(ax)=f_1(a)F(x)$ for $a\in I$ and $x\in P$, 
\item $F_1(Q)$ generates $P$ as an $S$-module.
\end{enumerate}
\end{Def}

\begin{Remark}
Here $F_1$ determines $F$. Indeed, if $f_1^\sharp(t)=1$,
for $x\in P$ we have $F(x)=F_1^\sharp(tx)$.
In particular $F(x)=\pi F_1(x)$ when $x$ lies in $Q$.
\end{Remark}

\begin{Remark}
If $(P,Q,F,F_1)$ is a window, 
there is a decomposition of $S$-modules $P=L\oplus T$ 
with $Q=L\oplus IT$, called normal decomposition, and
$$
\Psi:L\oplus T\xrightarrow{F_1+F}P
$$
is an $f$-linear isomorphism (which means that
$\Psi^\sharp$ is an isomorphism). 
Conversely, for given
finite free $S$-modules $L$ and $T$,
the set of window structures on 
$(P=L\oplus T,Q=L\oplus IT)$ is bijective to the
set of $f$-linear isomorphisms $\Psi$ as above.
\end{Remark}

\section{Functoriality}

Assume that $\FFF$ and $\FFF'$ are frames and
$u\in S'$ is a unit.

\begin{Def}
A $u$-morphism of frames $\alpha:\FFF\to\FFF'$ is 
a ring homomorphism $\alpha:S\to S'$ with
$\alpha(I)\subseteq I'$ such that 
$f'\alpha=\alpha f$ and $f_1'\alpha=u\alpha f_1$.
A morphism of frames is a $1$-morphism of frames.
\end{Def}

Let $\alpha:\FFF\to\FFF'$ be a $u$-morphism of frames.

\begin{Def}
If $\PPP$ and $\PPP'$ are windows 
over $\FFF$ and $\FFF'$, respectively,
an $\alpha$-morphism $g:\PPP\to\PPP'$ is a
homomorphism $g:P\to P'$ of $S$-modules with 
$g(Q)\subseteq Q'$ such that $F'g=gF$
and $F_1'g=ugF_1$.
\end{Def}

For every window $\PPP$ over $\FFF$ there is a base
change $\alpha_*\PPP$ over $\FFF'$
with an $\alpha$-morphism $\PPP\to\alpha_*\PPP$ such that 
$\Hom_\alpha(\PPP,\PPP')=\Hom_{\FFF'}(\alpha_*\PPP,\PPP')$.
Clearly this requirement determines the window $\alpha_*\PPP$ 
uniquely. It can be constructed explicitly as follows: 
If a normal decomposition $(L,T,\Psi)$ for $\PPP$ is
chosen, a normal decomposition for $\alpha_*\PPP$ is
$(S'\otimes_SL,S'\otimes_ST,\Psi')$
with $\Psi'(s'\otimes l)=uf'(s')\otimes\Psi(l)$
and $\Psi'(s'\otimes t)=f'(s')\otimes\Psi(t)$.


\section{Crystalline morphisms}

\begin{Def}
\label{DefCrys}
A morphism of frames $\alpha:\FFF\to\FFF'$ is
called crystalline if 
$
\alpha_*:(\text{windows over $\FFF$})
\to(\text{windows over $\FFF'$})
$
is an equivalence of categories.
\end{Def}


\begin{Thm}
\label{ThCrys}
Suppose a morphism of frames $\alpha:\FFF\to\FFF'$
induces an isomorphism $R\cong R'$ 
and a surjection $S\to S'$ with nilpotent
kernel $\Fa\subset S$ which has a filtration 
$\Fa=\Fa_0\supseteq\ldots\supseteq\Fa_n=0$
such that $f(\Fa_i)\subseteq\Fa_{i+1}$ and 
$f_1(\Fa_i)\subseteq\Fa_i$ and $f_1$ is 
elementwise nilpotent on $\Fa_i/\Fa_{i+1}$.
Then $\alpha$ is crystalline.
\end{Thm}

The proof is a variation of \cite{Zink-Disp}, Theorem 44
and \cite{Zink-DDisp}, Theorem 3.

\begin{proof}
Since $\alpha$ factors as 
$\FFF\to\FFF''\to\FFF'$ with $S''=S/\Fa_1$,
by induction we may assume that $f(\Fa)=0$ 
and that $f_1$ is elementwise nilpotent on $\Fa$.
We may also assume that $\Fa^2=0$ because 
the powers of $\Fa$ are stable under $f_1$.

The functor $\alpha_*$ is essentially surjective 
since the constituents of a normal
decomposition and the $f$-linear isomorphism $\Psi$ 
can be lifted from $\FFF'$ to $\FFF$.
In order that $\alpha_*$ is fully faithful
it suffices that $\alpha_*$ is fully faithful on 
automorphisms because a homomorphism $g:\PPP\to\PPP'$ 
can be encoded by the automorphism 
$\left(\begin{smallmatrix}1&0\\g&1\end{smallmatrix}\right)$
of $\PPP\oplus\PPP'$.
Since for a window $\PPP$ over $\FFF$ 
an automorphism of $\alpha_*\PPP$
can be lifted to an $S$-module automorphism of $P$ 
it suffices to prove the following assertion.

{\em
Assume that $\PPP=(P,Q,F,F_1)$ and $\PPP'=(P,Q,F',F_1')$
are two windows over $\FFF$ such that $F\equiv F'$ and
$F_1\equiv F_1'$ modulo $\Fa$. Then there is a unique
isomorphism $g:\PPP\cong\PPP'$ with 
$g\equiv\id$ modulo $\Fa$.
}

We write $F_1'=F_1+\eta$ and $F'=F+\varepsilon$ and $g=1+\omega$, 
where $\eta:Q\to\Fa P$ and $\varepsilon:P\to\Fa P$ are given, 
and $\omega:P\to\Fa P$ is an arbitrary homomorphism of $S$-modules.
The induced $g$ is an isomorphism of windows if and only if
$gF_1=F_1'g$ on $Q$, which translates into
\begin{equation}
\label{EqCond1}
\eta=\omega F_1-F_1\omega-\eta\omega.
\end{equation}
Here $\eta\omega=0$ because for $a\in\Fa$ and $x\in P$
we have $\eta(ax)=f_1(a)\varepsilon(x)$,
which is zero as $\Fa^2=0$. We fix a normal
decomposition $P=L\oplus T$ with $Q=L\oplus IT$.
For $l\in L$, $t\in T$, and $a\in I$ we have
\begin{gather*}
\eta(l+at)=\eta(l)+f_1(a)\varepsilon(t), \\
\omega(F_1(l+at))=\omega(F_1(l))+f_1(a)\omega(F(t)), \\
F_1(\omega(l+at))=F_1(\omega(l))+f_1(a)F(\omega(t)).
\end{gather*}
Here $F\omega=0$ because for $a\in\Fa$ and $x\in P$
we have $F(ax)=f(a)F(x)$, and $f(\Fa)=0$.
Hence \eqref{EqCond1} is equivalent to:
\begin{equation}
\label{EqCond2}
\left\{
\begin{array}{l}
\varepsilon=\omega F\qquad\qquad\;\text{ on } T, \\
\eta=\omega F_1-F_1\omega\quad\text{ on } L.
\end{array}
\right.
\end{equation}

Since $\Psi:L\oplus T\xrightarrow{F_1+F} P$ is an $f$-linear
isomorphism, the datum of $\omega$ is equivalent to 
the pair of $f$-linear homomorphisms
$$
\omega_T=\omega F:T\to\Fa P,\quad
\omega_L=\omega F_1:L\to\Fa P.
$$
Let $\lambda:L\to L^{(1)}$ be the composition
$L\subseteq P\xrightarrow{(\Psi^\sharp)^{-1}}
L^{(1)}\oplus T^{(1)}\xrightarrow{pr_1}L^{(1)}$
and let $\tau:L\to T^{(1)}$ be analogous
with $pr_2$ in place of $pr_1$. 
Then \eqref{EqCond2} becomes:
\begin{equation}
\label{EqCond3}
\left\{
\begin{array}{l}
\omega_T=\varepsilon|_T, \\
\omega_L-F_1\omega_L^\sharp\lambda=\eta|_L+F_1\omega_T^\sharp\tau.
\end{array}
\right.
\end{equation}
Let $U(\omega_L)=F_1\omega_L^\sharp\lambda$.
The endomorphism $F_1$ of $\Fa P$ is elementwise
nilpotent because $F_1(ax)=f_1(a)F(x)$ and because
$f_1$ is elementwise nilpotent on $\Fa$ by assumption.
Since $L$ is finitely generated it follows that
$U$ is elementwise nilpotent, so $1-U$ is bijective, 
and \eqref{EqCond3} has a unique solution.
\end{proof}


\section{Abstract deformation theory}

The Hodge filtration of a window $\PPP$
is the submodule $Q/IP\subseteq P/IP$.

\begin{Lemma}
\label{LeHodge}
Let $\alpha:\FFF\to\FFF'$ be a morphism
of frames such that $S=S'$, i.e.\ $I\subseteq I'$
is a sub-ideal and $f_1'$ is an extension of $f_1$. 
Then the base change functor induces an equivalence
between the category of windows over $\FFF$ and 
the category of pairs consisting of a window 
$\PPP'$ over $\FFF'$ and a lift of its Hodge 
filtration to a direct summand $V\subseteq P'/IP'$.
\end{Lemma}

\begin{proof}
Let $\PPP'$ over $\FFF'$ together with $V\subseteq P'/IP'$ 
be given and let $Q\subset P'$ be the inverse image of $V$.
In order that $(P',Q,F',F_1'|_Q)$ is a window
we must show that $F_1'(Q)$ generates $P'$. If
$P'=L'\oplus T'$ such that $Q=L'\oplus IT'$, this is
equivalent to $F_1'+F':L'\oplus T'\to P'$ being an 
$f$-linear isomorphism, which holds because
$\PPP'$ is a window.
\end{proof}

Assume that a morphism of frames $\alpha:\FFF\to\FFF'$
is given such that $S\to S'$ is surjective with nilpotent
kernel $\Fa$ and $I'=IS'$.
We want to factor $\alpha$ into morphisms of frames
$$
(S,I,R,f,f_1)\xrightarrow{\alpha_1}
(S,I'',R',f,f_1'')\xrightarrow{\alpha_2}
(S',I',R',f',f_1')
$$ 
such that $\alpha_2$ satisfies the hypotheses of
Theorem \ref{ThCrys}. Necessarily $I''=I+\Fa$. 
The main point is to define $f_1'':I''\to S$, which is 
equivalent to define an $f$-linear homomorphism $f_1'':\Fa\to\Fa$ 
that extends the restriction of $f_1$ to $\Fa\cap I$
and satisfies the hypotheses of Theorem \ref{ThCrys}.

If this is achieved, Theorem \ref{ThCrys} and
Lemma \ref{LeHodge} show that windows over
$\FFF$ are equivalent to windows $\PPP'$ over $\FFF'$
together with a lift of the Hodge filtration
to a direct summand of $P/IP$, where
$\PPP''=(P,Q'',F,F_1'')$ is the unique lift 
of $\PPP'$ under $\alpha_2$.


\section{Dieudonn\'e frames}

Let $R$ be a noetherian complete local ring with 
maximal ideal $\Fm$ and perfect residue field 
$k$ of characteristic $p$. 
If $p=2$ we assume that $pR=0$. 
There is a unique subring $\WW(R)\subset W(R)$ 
stable under $f$ such that the projection $\WW(R)\to W(k)$
is surjective with kernel $\hat W(\Fm)$,
the ideal of all Witt vectors in $W(\Fm)$ 
whose coefficients converge to zero $\Fm$-adically. 
Let $\II_R$ be the kernel of the
natural surjection $\WW(R)\to R$.

The Dieudonn\'e frame associated to $R$ is
$$
\FFF_R=(\WW(R),\II_R,R,f,f_1)
$$
with $f_1=v^{-1}$; we note that $\WW(R)$ is a local ring. 
In this case $\pi=p$.
Windows over $\FFF_R$ are Dieudonn\'e displays
over $R$ in the sense of \cite{Zink-DDisp}.
A local homomorphism $R\to R'$ induces a morphism 
of frames $\FFF_R\to\FFF_{R'}$.

Suppose $R'=R/\Fb$ for a nilpotent ideal $\Fb$ 
equipped with elementwise nilpotent divided powers.
The projection $\WW(R)\to\WW(R')$ is surjective with 
kernel $\hat W(\Fb)=W(\Fb)\cap\hat W(\Fm)$.
The divided Witt polynomials define an isomorphism
of $\WW(R)$-modules
$$
\log:\hat W(\Fb)\cong\Fb^{<\infty>}
$$
where $\Fb^{<\infty>}$ is the group of
sequences $[b_0,b_1,\ldots]$ with $b_i\in\Fb$
which converge to zero $\Fm$-adically,
and $x\in\WW(S)$ acts by 
$[b_0,b_1,\ldots]\mapsto [w_0(x)b_0,w_1(x)b_1,\ldots]$.
In logarithmic coordinates, the restriction 
of $f_1$ to $\hat W(\Fb)\cap\II_R$ is given by
$$
f_1[0,b_1,b_2,\ldots]=[b_1,b_2,\ldots].
$$
Let $\tilde\II=\II_R+\hat W(\Fb)$.
Then $f_1$ extends uniquely to an $f$-linear homomorphism 
$$
\tilde f_1:\tilde\II\to\WW(R)
$$
with $\tilde f_1[b_0,b_1,\ldots]=[b_1,b_2,\ldots]$
on $\hat W(\Fb)$, and we obtain a factorisation
\begin{equation}
\label{EqDieuDef}
\FFF_R\xrightarrow{\alpha_1}
\FFF'=(\WW(R),\tilde\II,f,\tilde f_1)
\xrightarrow{\alpha_2}\FFF_{R'}.
\end{equation}
The following is a reformulation of \cite{Zink-DDisp}, Theorem 3.

\begin{Prop}
\label{PrCrysDieu}
Here $\alpha_2$ is crystalline.
\end{Prop}

It follows that deformations of Dieudonn\'e displays
from $R'$ to $R$ are classified by lifts of the Hodge
filtration; this is \cite{Zink-DDisp}, Theorem 4.

\begin{proof}
When $\Fm$ is nilpotent, $\alpha_2$ satisfies the
hypotheses of Theorem \ref{ThCrys};
the required filtration of $\Fa=\hat W(\Fb)$ 
is $\Fa_i=p^i\Fa$. In general, the hypotheses
of Theorem \ref{ThCrys} are not satisfied 
because $f_1:\Fa\to\Fa$ is only topologically nilpotent.
However, one can find a sequence of ideals
$R\supset I_1\supset I_2\ldots$ which define the
$\Fm$-adic topology such that each $\Fb\cap I_n$ is
is stable under the divided powers of $\Fb$. 
Then the proposition holds for each $R/I_n$ in place of $R$,
and the general case follows by passing to the projective limit.
\end{proof}


\section{The Breuil frames}

We return to the notation fixed in the introduction,
in particular $p\geq 3$.
Let $I_a=E\FS_a$ be the kernel of $\FS_a\to R_a$;
note that $E$ is not a zero divisor in $\FS_a$.
The Frobenius $f$ of $W(k)$ is extended to $\FS_a$
by $f(u)=u^p$ and $f(t_i)=t_i^p$. For $x\in I_a$
let $f_1(x)=f(x/E)$. Then
$$
\BBB_a=(\FS_a,I_a,R_a,f,f_1)
$$
is a frame with $\pi=f(E)$.
Windows over $\BBB_a$ 
in the sense of Definition \ref{DefWin} 
are (equivalent to) the Breuil windows
relative to $\FS_a\to R_a$ introduced in \cite{Vasiu-Zink};
see loc.cit., Lemma 1. 

The frames $\BBB_a$ are related with Dieudonn\'e frames
as follows. There is a unique ring homomorphism 
$\varkappa:\FS\to\WW(R)$ that lifts the projection
$\FS\to R$ and commutes with $f$.
By \cite{Vasiu-Zink}, Lemma 2, we have
$\varkappa(f(E))=pu$ for a unit $u\in R$.
It is easy to see that $\varkappa$ induces 
compatible $u$-morphisms of frames 
$\varkappa_a:\BBB_a\to\FFF_{R_a}$
lying over the identity of $R_a$.

\begin{Thm}
\label{ThMain}
For all $a\in\NN\cup\{\infty\}$ the morphism
$\varkappa_a$ is crystalline.
\end{Thm}

The case $a=1$ follows from \cite{Zink-Windows},
and $a=\infty$ is the main result of \cite{Vasiu-Zink}.

\begin{proof}
For $a\in\NN$ we construct 
a commutative diagram of frames such that
vertical arrows are $u$-morphisms and horizontal
arrows are $1$-morphisms:
\begin{equation}
\label{EqDiagFrames}
\xymatrix{
\BBB_{a+1} \ar[r]^{\beta_1} \ar[d]^{\varkappa_{a+1}} &
\BBB' \ar[r]^{\beta_2} \ar[d]^{\varkappa'} &
\BBB_{a} \ar[d]^{\varkappa_a} \\
\FFF_{R_{a+1}} \ar[r]^{\alpha_1} &
\FFF' \ar[r]^{\alpha_2} &
\FFF_{R_a}
}
\end{equation}
The lower line is the factorisation \eqref{EqDieuDef} 
of the projection $\FFF_{R_{a+1}}\to\FFF_{R_a}$
with respect to the trivial divided powers 
on the kernel $\Fb=p^aR_{a+1}$.
In particular, $\alpha_2$ is crystalline by 
Proposition \ref{PrCrysDieu}.

Let $\Fa=u^{ae}\FS_{a+1}=\Ker(\FS_{a+1}\to\FS_a)$.
We define a frame
$$
\BBB'=(\FS_{a+1},I',R_a,f,f_1')
$$
where $I'=I_{a+1}+\Fa$, and where $f_1':I'\to\FS_{a+1}$ 
is the unique extension of $f_1$ with $f_1'(\Fa)=0$.
This is well-defined because $f_1=0$ on 
$\Fa\cap I_{a+1}=u^{ae}E\FS_{a+1}$.
Hence the upper line of \eqref{EqDiagFrames}
is constructed too. Here $\beta_2$ is crystalline
by Theorem \ref{ThCrys};
the required filtration of $\Fa$ is trivial.

In order that $\varkappa'$, necessarily given by
$\varkappa_{a+1}$, is a $u$-morphism of frames,
we need that $f_1(\varkappa_{a+1}(u^{ae}x))=0$
for $x\in\FS_{a+1}$. 
Now $\varkappa_{a+1}(u^{ae})$ is the Teich\-m\"uller
element $[u^{ae}]$ and 
$\log([u^{ae}])=[u^{ae},0,0,\ldots]$.
Hence $f_1([u^{ae}])=0$ and thus
$f_1(\varkappa_{a+1}(u^{ae}x))=
f_1(\varkappa_{a+1}(u^{ae}))f(\varkappa_{a+1}(x))=0$
as required.

It follows that lifts of windows from 
$\BBB_a$ to $\BBB_{a+1}$ or from 
$\FFF_{R_a}$ to $\FFF_{R_{a+1}}$ are both
classified by lifts of the Hodge filtration
in the same way.
Since $\varkappa_{1}$ is crystalline by \cite{Zink-Windows}, 
by induction $\varkappa_{a}$ is crystalline for all finite $a$,
so $\varkappa_{\infty}$ is crystalline
by passing to the projective limit.
\end{proof}

\begin{Remark}
Along the same lines one can show directly that
$\varkappa_{1}$ is crystalline.
Namely, for each $\underline n=(n_1,\ldots, n_{r+1})$
with $n_i\in\NN\cup\infty$ let 
$$
S_{\underline n}=W(k)[[t_1,\ldots,t_{r+1}]]/
(t_1^{n_1},\ldots,t_{r+1}^{n_{r+1}})
$$
with $t_i^\infty=0$. 
Let $I_{\underline n}=pS_{\underline n}$
and $R_{\underline n}=S_{\underline n}/I_{\underline n}$.
We have a frame
$$
\CCC_{\underline n}=
(S_{\underline n},I_{\underline n},R_{\underline n},f,f_1)
$$
where $f(t_i)=t_i^p$ and $f_1(x)=f(x/p)$ 
for $x\in I_{\underline n}$,
and compatible morphisms of frames 
$\varkappa_{\underline n}:\CCC_{\underline n}
\to\FFF_{R_{\underline n}}$. 
Mutatis mutandis the proof of Theorem \ref{ThMain}
shows that these are all crystalline;
in the initial case $\underline n=(1,\ldots 1)$ 
the morphism $\varkappa_{\underline n}$ is an isomorphism.
In $\FS_1$ we have $f(E)=vp$ for a unit $v$,
and for $\underline n=(\infty,\ldots,\infty,e)$ 
there is a $v^{-1}$-isomorphism of frames
$\CCC_{\underline n}\cong\BBB_1$ compatible with
the respective $\varkappa$'s. The assertion follows.
\end{Remark}

\end{document}